\documentclass[12pt]{article}
\usepackage{amsmath,amsthm,amsfonts,amssymb,latexsym,enumerate,graphicx,mathrsfs}
\usepackage[utf8]{inputenc}
\usepackage[letterpaper]{geometry}
\usepackage{hyperref}
\hypersetup{colorlinks, citecolor=blue, linkcolor=black}

\usepackage{todonotes}

\usepackage{tikz}
\usetikzlibrary{arrows,decorations.pathmorphing,decorations.pathreplacing}

\newtheorem{thm}{Theorem}[section]
\newtheorem{cor}[thm]{Corollary}
\newtheorem{prop}[thm]{Proposition}

\theoremstyle{definition}

\theoremstyle{remark}
\newtheorem{remark}[thm]{Remark}

\newcommand{\define}{\textit}

\newcommand {\calB} {{\mathcal {B}}}

\newcommand {\calP} {{\mathcal {P}}}   
   
\newcommand {\calR} {{\mathcal {R}}}

\newcommand {\bbF} {{\mathbb {F}}}   
   
\newcommand {\bbH} {{\mathbb {H}}}

\newcommand {\bbR} {{\mathbb {R}}}

\newcommand{\bk}[1]{{\left\langle #1 \right\rangle}}

\newcommand{\cay}[2]{\ensuremath{\mathbf{Cay}_{#1}\left(#2 \right)}}
\newcommand{\bdy}{\ensuremath{\partial_{\infty}}}

\newcommand{\raybundle}[2]{\ensuremath{\calR\calB}\left(#1,#2\right)}
\newcommand{\prescx}[1]{\ensuremath{\calP\left(#1\right)}}
\newcommand{\ladder}{\ensuremath{\mathfrak{L}}}
\newcommand{\piece}{\ensuremath{\mathfrak{p}}}
\newcommand{\sline}{\ensuremath{\mathfrak{s}}}

\author{Nicholas Touikan}
\title{On geodesic ray bundles in hyperbolic groups}

\begin{document}
\maketitle
\abstract{We construct a Cayley graph $\cay S \Gamma$  of a hyperbolic
  group $\Gamma$ such that there are elements $g,h\in\Gamma$ and a
  point $\gamma \in \bdy\Gamma = \bdy\cay S\Gamma$ such that
  the sets $\raybundle g \gamma$ and $\raybundle h \gamma$ in
  $\cay S \Gamma$ of vertices along geodesic rays from $g,h$ to
  $\gamma$ have infinite symmetric difference; thus answering a
  question of Huang, Sabok and Shinko.}

\section{Introduction}

To every infinite finite valence tree $T$ we can associate a boundary
at infinity $\bdy T$ corresponding to ends of infinite rays. $\bdy T$
is homeomorphic to a Cantor set. A metric space is called
$\delta$-hyperbolic if, roughly speaking, up to an error term $\delta$
it has a tree-like structure. Analogously to a tree, to a
$\delta$-hyperbolic space $X$, one can assign a Gromov boundary
$\bdy X$ which is a compact, metrizable, yet oftentimes exotic set,
corresponding to equivalence classes of ends of infinite
\emph{geodesic} rays. A group $\Gamma$ is called \define{hyperbolic}
if one of its Cayley graphs is is a $\delta$-hyperbolic metric space
for some $\delta\geq 0$. In this case to $\Gamma$ we can assign a
canonical Gromov boundary $\bdy \Gamma$ on which $\Gamma$ acts
non-trivially. The deep connections between the properties of
$\bdy\Gamma$ and the group $\Gamma$ makes it highly a structured, and
therefore fascinating, object to study.

In \cite{huang2017hyperfiniteness} Huang, Sabok and Shinko investigate
Borel equivalence relations on $\bdy\Gamma$. They show that if
$\Gamma$ is a hyperbolic group with the additional property that
$\Gamma$ acts properly discontinuously and cocompactly on a CAT(0)
cube complex, i.e. $\Gamma$ is \define{cubulated}, then the action of
$\Gamma$ on its boundary $\bdy \Gamma$ is hyperfinite. This
generalizes a result of Dougherty, Jackson and Kechris \cite[Corollary
8.2]{dougherty1994structure} from the class of free groups to the
larger class of cubulated hyperbolic groups.

Although the result of \cite{huang2017hyperfiniteness} feels like it
should be true for all hyperbolic groups, an additional cubulation
requirement is needed to prove a key lemma, \cite[Lemma
1.3]{huang2017hyperfiniteness}, which states that for any two vertices
$x,y$ of a $\delta$-hyperbolic CAT(0) cube complex $C$ and for any
point $\gamma \in \bdy C$ the sets, called \define{ray bundles},
$\raybundle{x}{\gamma}$ and $\raybundle{y}{\gamma}$ of vertices of $C$
that occur along geodesic rays from $x$ and $y$ (respectively) to
$\gamma \in \bdy C$ have finite symmetric difference.

The authors pose \cite[Question 1.4]{huang2017hyperfiniteness} which
asks if \cite[Lemma 1.3]{huang2017hyperfiniteness} holds for any
Cayley graph of a hyperbolic groups. Not only would a positive answer
immediately imply that the action of any hyperbolic group $\Gamma$ on
$\bdy \Gamma$ is hyperfinite, but this is also a very natural question
to ask from the point of view of geometric group theory. This paper
gives a negative answer by giving examples of Cayley graphs
$\cay{S}{\Gamma}$ of hyperbolic groups $\Gamma$ with vertices $x,y$
and some $\gamma \in \bdy\cay{S}{\Gamma}$ such that the ray bundles
$\raybundle{x}{\gamma}$ and $\raybundle{y}{\gamma}$ have infinite
symmetric difference. This example, if anything, reinforces the
relevance of \cite[Lemma 1.3]{huang2017hyperfiniteness}.

The methods of this paper will be familiar to geometric group
theorists, but, since this paper is aimed at a broader audience,
necessary background is included to make it self-contained. That being
said, the reader is expected to know the following notions from
topology: group presentations, fundamental groups, the Seifert-van
Kampen theorem, and universal covering spaces.

\subsection{Acknowledgements}

I first wish to thank Michael Hull and Jindrich Zapletal for the
invitation to the South Eastern Logic Symposium 2017, which greatly
increased my appreciation of the contemporary work of descriptive set
theorists. I also wish to thank Marcin Sabok for posing this question
about symmetric differences of ray bundles, specifically about the
embedability of bad ladders into Cayley graphs, and for an interesting
discussion, encouragement and feedback. Finally I am grateful to Bob Gilman and
Paul Schupp for conversations that confirmed that the main result of
this paper is probably not a trivial consequence of what is known
about the automaticity of the language of geodesics in hyperbolic
groups.

\section{Hyperbolic groups and their boundary}

The author recommends \cite{notes-on-hyp-groups} for an accessible yet
thorough treatment of the topics in this section. Given a group $\Gamma$
and a generating set $S$ of $\Gamma$ we can construct a Cayley graph
$\cay S \Gamma$ which is a directed graph whose vertices are the elements
of $\Gamma$ and for each $g \in \Gamma$ and $s \in S$ we draw the edge\[
\begin{tikzpicture}
  \draw[fill=black] (0,0)node[left]{$g$} circle (0.1)
  --node{$\blacktriangleright$} node[above]{$s$} (2,0) circle (0.1) node[right]{$gs$};
\end{tikzpicture}.\] By declaring each edge to be an isometric copy of
the closed unit interval, we make graphs into
connected metric spaces via the path metric. If $X$ is a graph we say
that a path starting at a vertex $v$ and ending at a vertex $u$ is
\define{geodesic} if it is the shortest possible path between
$u,v$. Typically there will be multiple geodesics between a pair of
vertices. A metric space is \define{$\delta$-hyperbolic} if it has the following property: for any three vertices
$u,v,w$ if $\alpha,\beta$, and $\gamma$ are geodesics from $u$ to $v$,
$v$ to $w$, and $w$ to $u$ respectively then $\alpha$ is contained in
a $\delta$-neighbourhood of $\beta \cup \gamma$. If a group $\Gamma$
has a $\delta$-hyperbolic Cayley graph with respect to one finite
generating set, then for any other finite generating set the
corresponding Cayley graph will also be $\delta'$-hyperbolic, though
possibly with $\delta'\neq\delta$. Such a group will therefore be
called a \emph{hyperbolic group.}

For example, if $A$ is a finite set of symbols and $\bbF(A)$ is the
free group on $A$, then, taking $A$ as a generating set of $\bbF(A)$,
the Cayley graph $\cay{A}{\bbF(A)}$ is a regular tree with valence
$|A|$ and in particular for any geodesics $\alpha,\beta$, and $\gamma$
as above, $\alpha \subset \beta \cup \gamma$ so that
$\cay{A}{\bbF(A)}$ is in fact 0-hyperbolic.



Let us now give a precise definition of the Gromov boundary
$\bdy \Gamma$. Let $S$ be a finite generating set of $\Gamma$. A
\define{geodesic ray} is a continuous map
\[\rho:[0,\infty) \to \cay S \Gamma\] such that for every pair of positive integers
$m<n$, $\rho(m)$ is a vertex and the segment $\rho([n,m])$ is a
geodesic. $\bdy \Gamma$ is the set of geodesic rays of $\cay S \Gamma$
modulo the
relation: $\rho \sim \rho'  \Leftrightarrow$ there is some $R\geq 0$
such that $\rho\left([0,\infty) \right)$ is contained in an
$R-$neighbourhood of $\rho'\left([0,\infty) \right)$ and
$\rho'\left([0,\infty) \right)$ is contained in an
$R-$neighbourhood of $\rho\left([0,\infty) \right)$.


We recommend the following exercises:
\begin{itemize}
\item If $\Gamma = \bbF(A)$ is a free group as above, then  $\partial
  \Gamma$ is naturally identified with a Cantor set.
\item If $\Gamma = \bk{a}\oplus\bk{b}$, the free abelian group of rank two
  (which is not a hyperbolic group), then $\bdy \cay A \Gamma$ can be
  identified with the circle at infinity for $\bbR^2$, but the action
  of $\Gamma$ (induced by translating rays) yields a trivial action on
  $\bdy\Gamma$.
\end{itemize}

That $\bdy \Gamma$, thus given, is well-defined, non-trivial,
canonical for $\Gamma$, and admits a non-trivial $\Gamma$ action is a
consequence of $\delta$-hyperbolicity. The reader may consult
\cite[\S6-\S8]{ghys1990groupes} or \cite{kapovich2002boundaries} for a
complete treatment of the topic.

\section{The bad ladder}\label{sec:bad-ladder}

Consider the infinite graph $\ladder$ consisting of two \define{sides}, copies
of $\bbR$, with a vertex at each integer point, and countably many
\define{rungs}, edges connecting vertices at corresponding
integral vertices on each side. Add a vertex to the middle of each
rung. The resulting graph $\ladder$ is shown in Figure
\ref{fig:ladder}.

\begin{figure}[htb]
  \centering
  \includegraphics[width=0.8\textwidth]{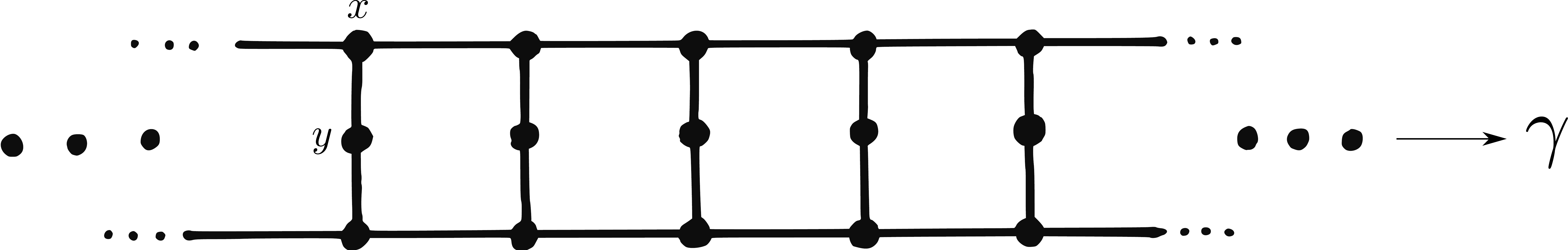}
  \caption{A ladder with a vertex $x$ on a side and a vertex $y$ in
    the middle of a rung.}
  \label{fig:ladder}
\end{figure}

We note that any two geodesic rays either go to the left or to the
right, and if they go in the same direction, they remain at a bounded
distance. It follows that $\bdy \ladder$ consists of two points.

\begin{prop}\label{prop:ladders-are-bad}
  Let $x$ be a vertex on a side of $\ladder$, let $y$ be a vertex in
  the middle of a rung and let $\gamma \in \bdy \ladder = \{\pm\infty\}$
  correspond to one of the ends of the ladder. Then the sets
  $\raybundle x  \gamma$ and $\raybundle y \gamma$ have infinite symmetric
  difference. 
\end{prop}
\begin{proof}
  Without loss of generality we may assume that $\gamma$ corresponds
  to $+\infty$. As any geodesic $\rho$ travels towards $\gamma$ it
  must eventually stay within one of the sides of $\ladder$. If $\rho$
  originates at $x$ then it is allowed to travel once through a rung
  to reach the other side. It follows that every vertex on a rung that
  is ``greater'' than $x$ is in $\raybundle x \gamma$. If $\rho$
  originates at $y$ in the middle of a rung, then once it enters a
  side $s_1$ it is no longer able to switch because if that happens
  then there is some initial segment $\rho'$ of $\rho$ whose length
  does not realize the distance between $y$ and the first point it
  encounters in $s_2\neq s_1$. See Figure \ref{fig:ladder}. It follows
  that $\raybundle y \gamma$ doesn't contain any vertices contained in
  rungs; thus the two sets have infinite symmetric difference.
\end{proof}

Although $\ladder$ is a hyperbolic graph, due to its nonhomogeniety,
it cannot be the Cayley graph of a group. We will now construct the
Cayley graph of a group, in fact a free group, which contains a ladder
$\ladder$ as a convex subgraph. That is to say any geodesic connecting
two points on the ladder inside this larger graphs must stay within
the ladder. To show this we must reach a sufficiently complete
understanding of the geometry of a Cayley graph. Although it is not
invoked explicitly, the proof is informed by the Bass-Serre theory of
groups acting on trees and corresponding decompositions into graphs of
spaces \cite{scott-wall,serre}.

\section{Embedding bad ladders into Cayley graphs}

We will take some liberties with notation and identify group
presentations with the groups they present. First consider the
presentation
\[\Gamma_0 = \bk{p,q,t \mid t^{-1}ptq^{-1}}\approx \bbF_2.\] For any
group presentation, there is a standard construction known as a
\define{presentation complex}, which is a CW-complex
$\prescx{\Gamma_0}$ obtained by gluing polygons (corresponding to
relations) to graphs (edges correspond to generators) in such a way
(as a consequence of the Seifert-van Kampen Theorem) that
$\pi_1\left(\prescx{\Gamma_0}\right) \approx
\Gamma_0$.

In this case presentation complex $\prescx{\Gamma_0}$ consists of a
graph with one vertex, three directed edges labelled $p,q,t$, and a
square along whose boundary the word $t^{-1}ptq^{-1}$ can be
read. This word specifies the identifying map between the boundary of
the square and a closed loop in the graph, making the latter
nullhomotopic. As a topological space $\prescx{\Gamma_0}$ can also be
obtained by taking a cylinder $A=[-1,1] \times S^1$, picking a
point on each boundary component and identifying them. This is shown
on Figure \ref{fig:Gamma0}.

\begin{figure}[htb]
  \centering
  \includegraphics[width=0.5\textwidth]{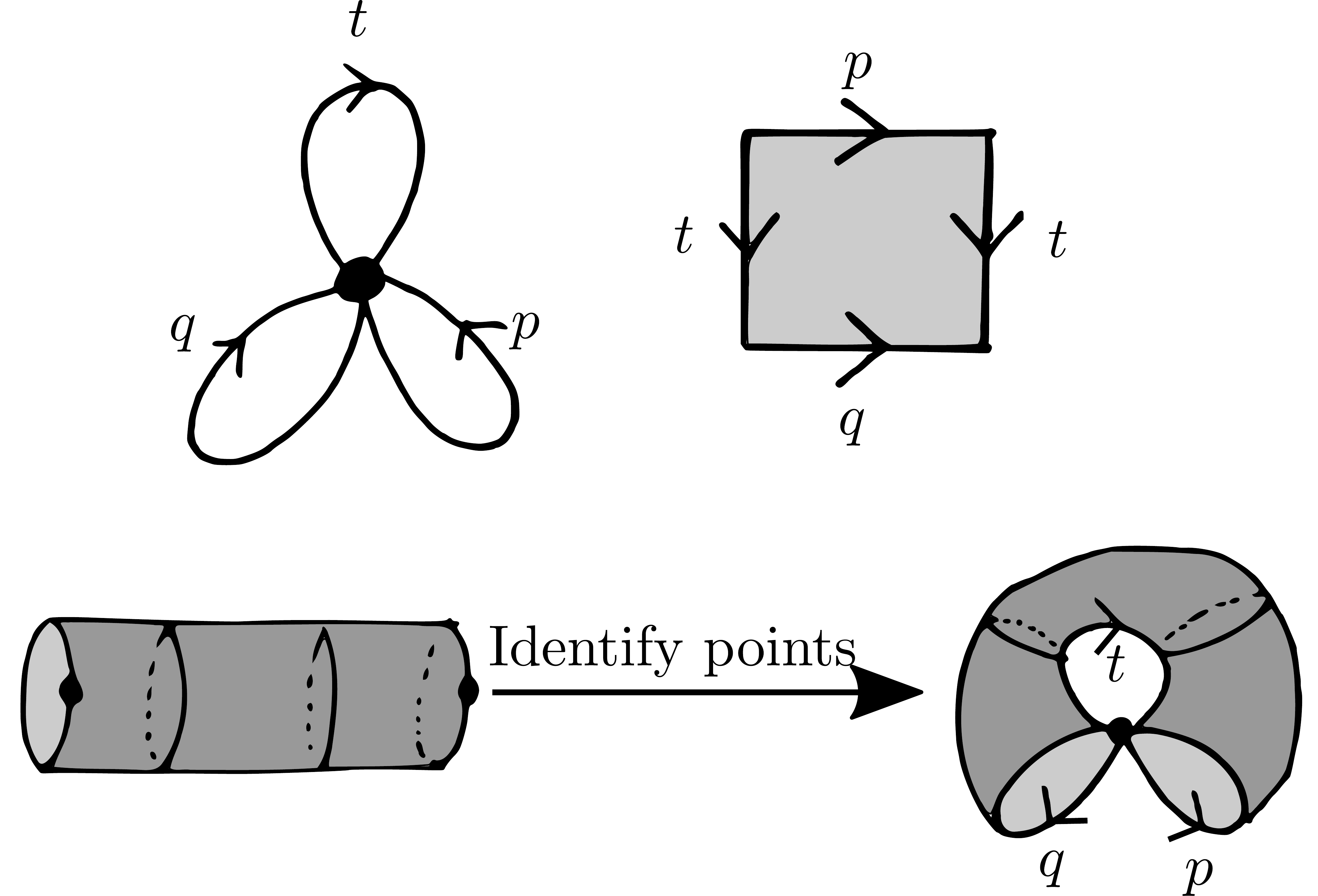}
  \caption{The presentation complex $\prescx{\Gamma_0}$}
  \label{fig:Gamma0}
\end{figure}

\begin{remark}
  The 1-skeleton of the universal cover
  $\widetilde{\prescx{\Gamma_0}}$ corresponds to the Cayley graph
  $\cay{\{p,q,t\}}{\bbF_2}$, i.e. the Cayley graph relative to the
  generating set explicitly given by the group presentation. This is
  true for any presentation complex.
\end{remark}
The universal cover $\widetilde{\prescx{\Gamma_0}}$ is a tree of
spaces obtained by taking an infinite collection of copies of strips
corresponding to connected components of the lift
$A\subset \prescx{\Gamma_0}$ in $\widetilde{\prescx{\Gamma_0}}$
attached by points. We call these \define{$pq$-strips}. This is shown
if Figure \ref{fig:pqstrips}. There is also a collection of
bi-infinite lines in $\widetilde{\prescx{\Gamma_0}}$ along which we
read $\ldots t t t \ldots$, we call these \define{$t$-lines.}
\begin{figure}[htb]
  \centering
  \includegraphics[width=0.5\textwidth]{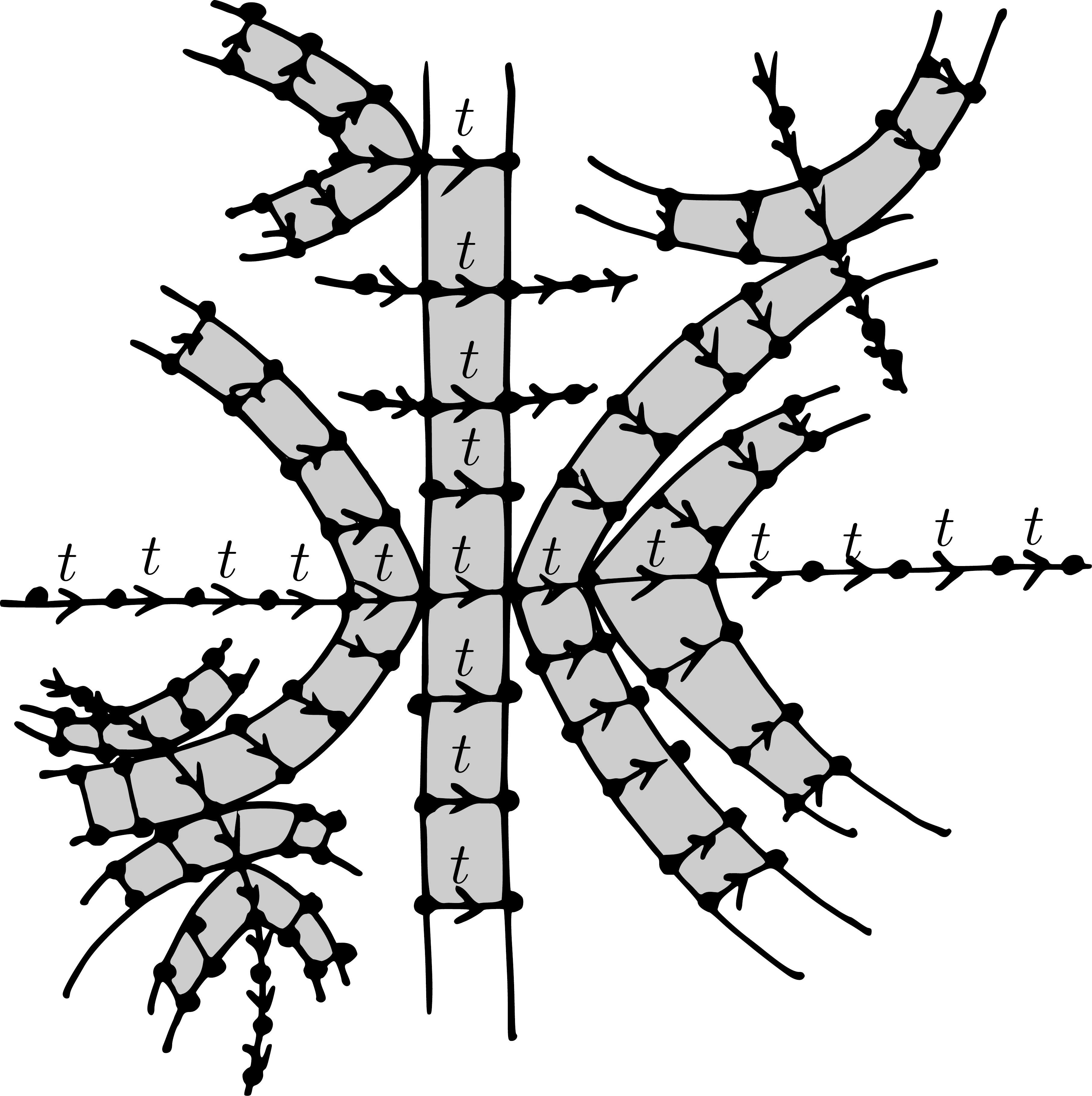}
  \caption{The universal cover of $\prescx{\Gamma_0}$. Decorated edges
    are labelled $t$. $pq$-strips are shaded grey.}
  \label{fig:pqstrips}
\end{figure}
Consider now the amalgamated free product:
\[\Gamma_1 = \Gamma_0 *_{t=s^2} \bk{s}  =  \bk{p,q,t,s \mid
    t^{-1}ptq^{-1}, s^{2}t^{-1}} \approx \bbF_2\] corresponding to
adjoining a square root $s$ to the basis element $t \in \bbF_2$. By
the Seifert-van Kampen Theorem, it can be realized as the fundamental
group of a space $\calP_1$, which is not a presentation complex,
obtained by taking a copy of $\prescx{\Gamma_0}$, a circle $C=S^1$,
and attaching another cylinder $D=[-1,1]\times S^1$ so that the attaching
map $\{-1\}\times S^1 \to \prescx{\Gamma_0}$ wraps with degree 1
around the loop corresponding to the edge with label $t$ and 
the other  attaching map $\{1\}\times S^1 \to C$ wraps with degree 2. See
Figure \ref{fig:adjoin-sqrt}.

\begin{figure}[htb]
  \centering
  \includegraphics[width=0.5\textwidth]{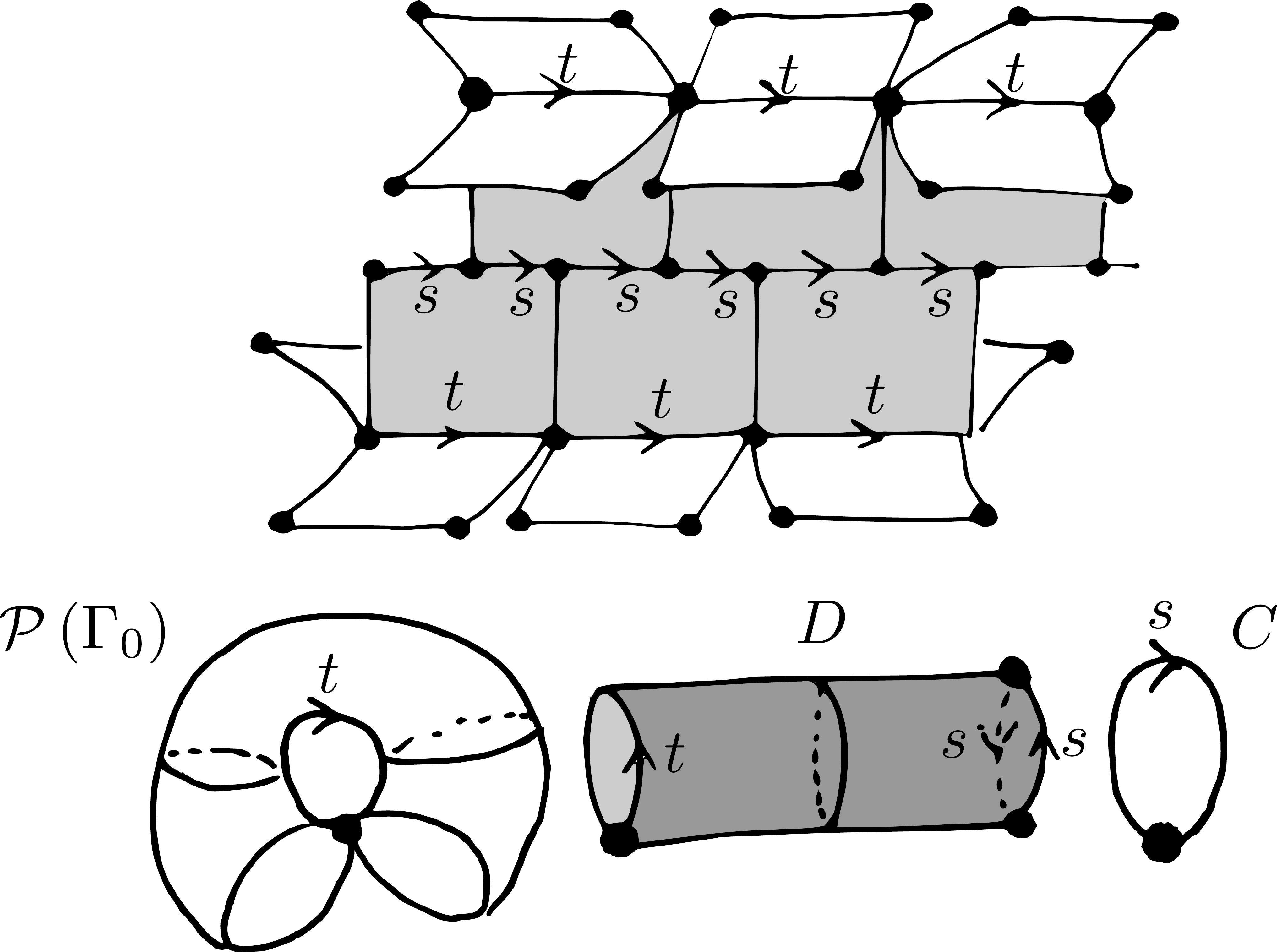}
  \caption{On top, a portion of the universal cover
    $\widetilde{\calP_1}$. Below, how the CW-complex $\calP_1$ is
    obtained from $\calP_0$.}\label{fig:adjoin-sqrt}
\end{figure}

In the universal covering space $\widetilde{\calP_1}$,
$\prescx{\Gamma_0} \subset \calP_1$ lifts to a countable collection of
disjoint copies of $\widetilde{\prescx{\Gamma_0}}$ called
\define{$\Gamma_0$-pieces} and the circle $C \subset \calP_1$ lifts to
a countable collection of disjoint lines called
\define{$C$-lines}. The connected components of lifts of the cylinder
$D$ are called \emph{$D-$strips}, copies of $[-1,1]\times \bbR$
connecting $t$-lines in $\Gamma_0$-pieces to
$C$-lines. In particular each $C$-line is attached to two
$D$-strips. Globally, the universal cover has the structure of a tree
of spaces. See Figure \ref{fig:tree-of-spaces}.

\begin{figure}[htb]
  \centering
  \includegraphics[width=0.33\textwidth]{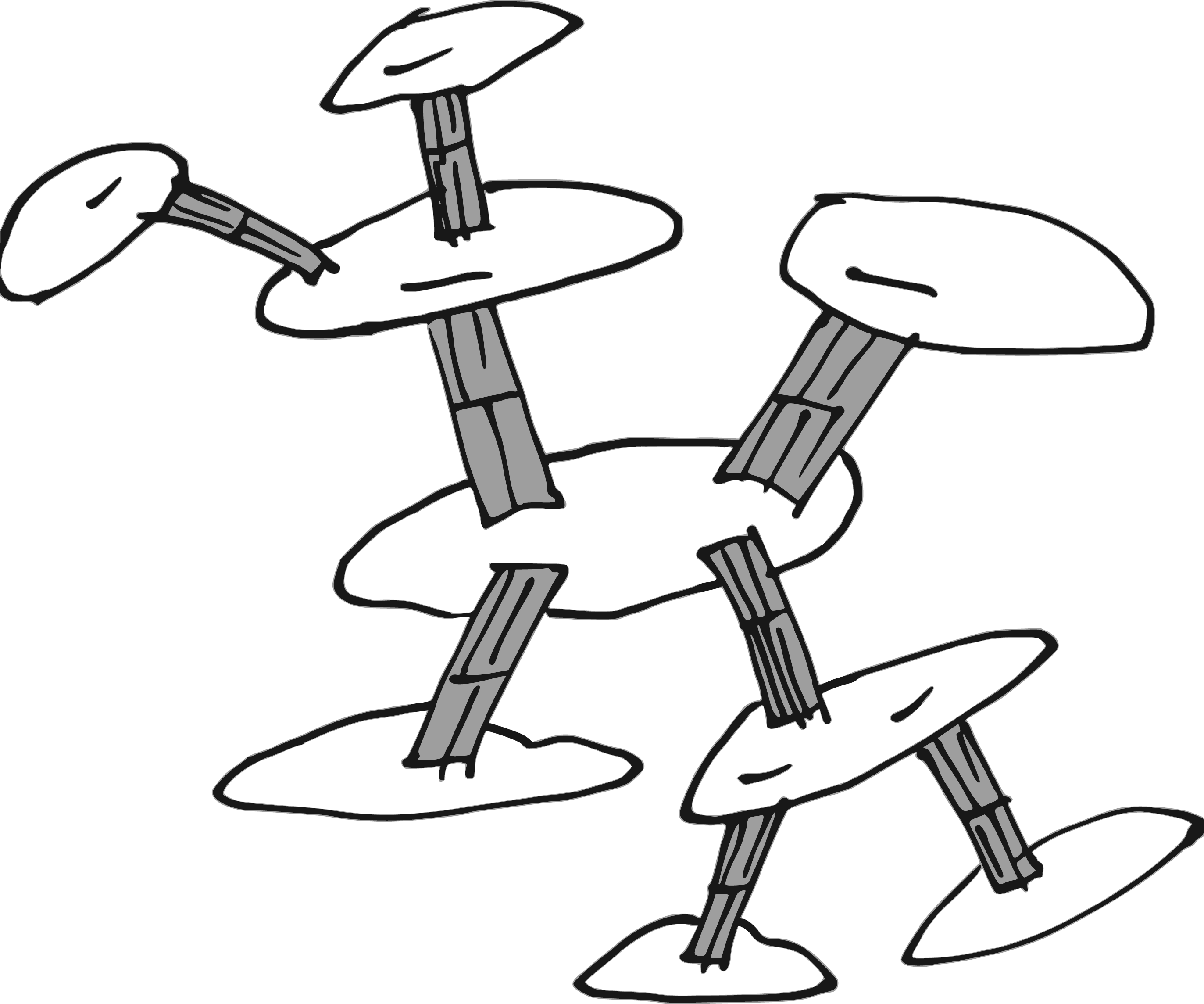}
  \caption{A portion of $\widetilde{\calP_1}$ depicted as a tree of
    spaces obtained by attaching $\Gamma_0$-pieces to $D$-strips
    (shown in grey) along $t$-lines. Note that although drawn as
    ``pancakes'' the $\Gamma_0$-pieces are actually copies of the
    space shown in Figure \ref{fig:pqstrips}.}
  \label{fig:tree-of-spaces}
\end{figure}

Our final presentation $\Gamma$ is obtained via the following
Tietze transformation:\[
  \Gamma_1 = \bk{p,q,t,s \mid t^{-1}ptq^{-1}, s^{2}t^{-1}} \approx
  \bk{p,q,s \mid s^{-2}ps^2q} = \Gamma.
\] This Tietze transformation corresponds to the fact that, since
$s^2=t$, we can remove $t$ from the generating set. Geometrically the
resulting presentation complex is obtained by collapsing the $[-1,1]$
factor in the cylinder $D=S^1\times[-1,1] \subset \calP_1$ to a
point. See Figure \ref{fig:collapse}.
\begin{figure}[htb]
  \centering
  \includegraphics[width=0.5\textwidth]{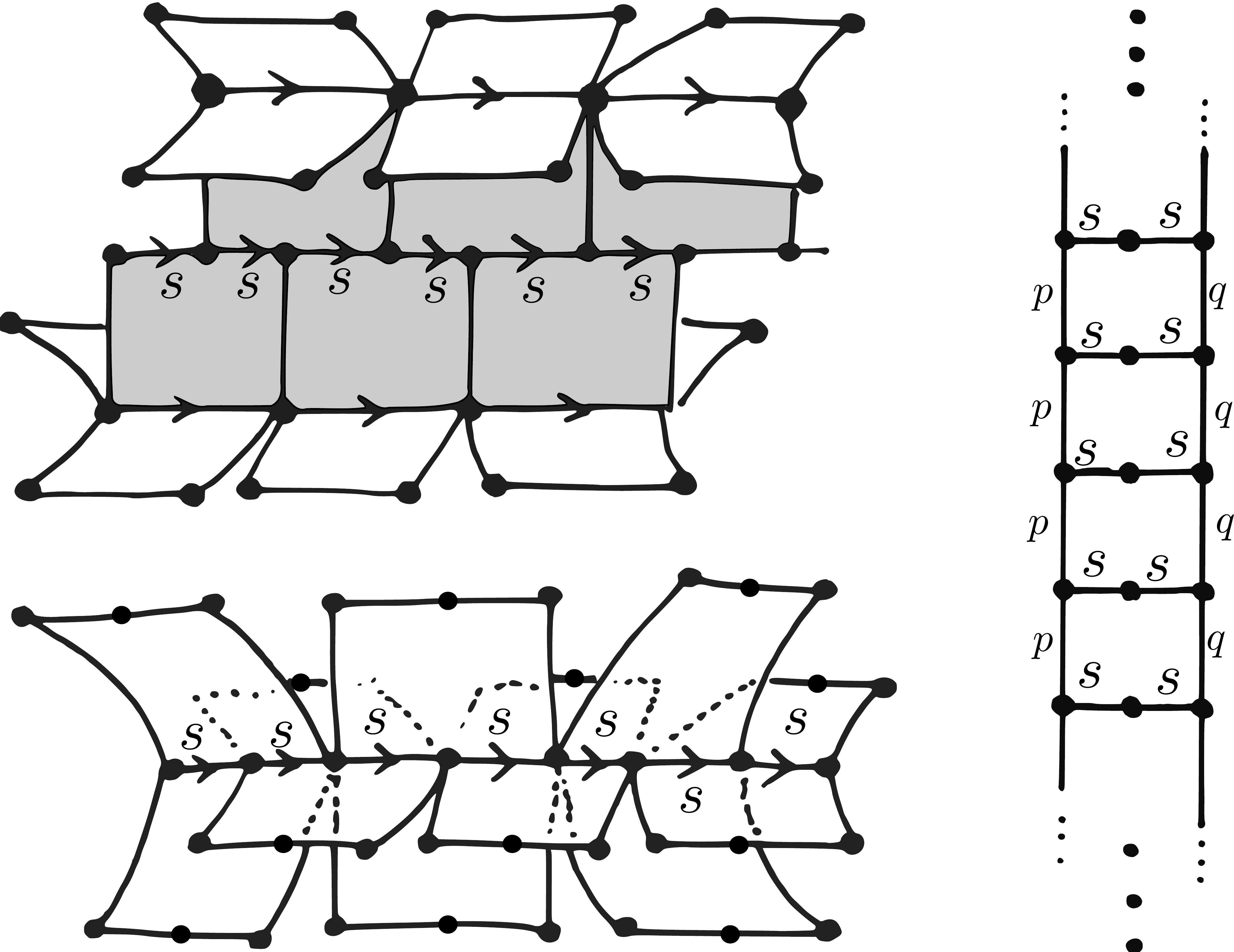}
  \caption{Collapsing $D$-strips (shaded grey) onto lines as seen from the universal
    cover, and the resulting $pq$-ladders, contained in a $\Gamma_0$-piece.}
  \label{fig:collapse}
\end{figure}
The universal cover of the presentation complex
$\widetilde{\prescx{\Gamma}}$ can be obtained by taking the
$\Gamma_0$-pieces in $\widetilde{\calP_1}$, subdividing each
$t$-labelled edge into a length 2 edge path labelled $ss$, replacing
$t-$lines with \define{$s$-lines}, and then identifying two $s$-lines
in different $\Gamma_0$-pieces if they were both connected by
$D$-strips to the same $C$-line. In this way $\widetilde{\calP_1}$ has
a large scale tree of spaces structure obtained taking resulting
$\Gamma_0$-pieces and attaching them along $s-$lines. Furthermore we
observe that each $\Gamma_0$-piece contains a ladder $\ladder$
obtained by gluing together squares labelled $s^{-2}ps^2q$ along
segments labelled $s^2$. We call such a ladder a
\define{$pq$-ladder}. See Figure \ref{fig:collapse}.

In this way $\widetilde{\prescx{\Gamma}}$ admits a depth 2
hierarchical decomposition as a tree of spaces. At the top level we
have $\Gamma_0$-pieces connected along $s$-lines as a tree of spaces,
then the $\Gamma_0$-pieces themselves are trees of $pq$-ladders,
connected by vertices.

\begin{prop}\label{prop:ladders-are-convex}
  A $pq$-ladder $\ladder$ is convex in the 1-skeleton of
  $\widetilde{\prescx{\Gamma}}$. Furthermore any geodesic ray starting
  in $\ladder$ and going to one of the ends of $\ladder$ must stay in $\ladder.$
\end{prop}
\begin{proof}
  Let $\piece$ be the $\Gamma_0$-piece containing a $pq$-ladder
  $\ladder$. Let $u,v \in \ladder$ be vertices and let $\rho$ be a
  geodesic connecting $u$ and $v$.

  \emph{Claim 1: $\rho$ cannot exit $\piece$.} Suppose towards the
  contrary that this was the case then, by the tree of spaces
  structure, $\rho$ must exit $\piece$ at some point $a$ contained in
  some $s$-line $\sline$, and then re-enter $\piece$ by at some other
  point $b\in \sline$ in the same $s$-line. It follows that if $\rho$
  is geodesic it cannot exit $\piece$ because the subsegment
  $\rho([n_a,n_b])$, where $\rho(n_a)=a, \rho(n_b)=b$, can be
  replaced the strictly shorter segment from $a$ to $b$ contained
  within $\sline$.

  \emph{Claim 2: If $\rho$ stays in the $\Gamma_0$-piece $\piece$, it
    cannot exit $\ladder$.} Indeed each piece consists of a tree of
  $pq$-ladders connected by points; since it is the same space as the
  one shown in Figure \ref{fig:pqstrips} except with each edge
  labelled $t$ replaced by a path of length 2 labelled $ss$. If $\rho$
  leaves $\ladder$ at some vertex $p$, then to re-enter $\ladder$ it
  must pass through $p$ again, contradicting that it is geodesic.

  The convexity of $\ladder$ now follows. This implies that any
  infinite path that stays in the $p$ or $q$ side of $\ladder$ is a
  geodesic ray. It remains to show that any geodesic ray starting at
  $x\in \ladder$ going to $\gamma \in \bdy \ladder$ stays in
  $\ladder$. Let $\rho$ be one such geodesic ray and let
  $\beta: [0,\infty) \to \widetilde{\prescx{\Gamma}}$ be another
  arc-length parameterized geodesic ray from $x$ to $\gamma$. Suppose
  that $\beta$ exits $\ladder$ at the point $\beta(N)$.

  By convexity of $\ladder$, $\beta$ cannot re-enter $\ladder$, but it
  could still travel close to it. By definition of the Gromov boundary
  there must be some bound $R$ such that for all $z$,
  $d(\beta(z),\rho) \leq R$. However, since $\beta$ is geodesic and
  arc-length parameterized, $d(\beta(N+M),\beta(N)) = M$ and since the
  shortest path from $\beta(N+M)$ to $\ladder$ must pass through
  $\beta(N)$ we conclude that
  \[d(\beta(N+M),\rho) \geq d(\beta(N+M),\ladder) = M.\] Since $\beta$
  is an infinite ray we may take $M>R$ which yields a contradiction.
\end{proof}

Proposition \ref{prop:ladders-are-bad} and
\ref{prop:ladders-are-convex} immediately imply the main result:

\begin{cor}\label{cor:main}
  Let $\Gamma = \bk{p,q,s \mid s^{-2}ps^2q}\approx \bbF_2$ and let
  $X = \cay{\{p,q,s\}}{\Gamma}$. If $\gamma \in \bdy X = \bdy \Gamma$
  corresponds to an end of a $pq$-ladder $\ladder$, $x$ is a vertex
  contained in a side of $\ladder$, and $y$ is a vertex contained in a
  rung of $\ladder$, then $\raybundle x \gamma$ and
  $\raybundle y \gamma$ have infinite symmetric difference.
\end{cor}

\subsection{A one-ended example}
The example we just gave is somewhat unsatisfying since it is a free
group. We will outline another construction, which was the original
example found by the author. This group is not free since it is
one-ended, which in the torsion-free case means it does not decompose
as a non-trivial free product. Consider first the presentation
\[ \Sigma_0 = \bk{a,b,p,q,t \mid abpa^{-1}b^{-1}q, ptqt^{-1}}.
\] This is an explicit decomposition of $\Sigma_0$ as an HNN extension
of a free group of rank 3 and the presentation complex is homeomorphic
to an orientable closed surface of genus 2. We then repeat the
construction in the previous section
\begin{align*}
 \Sigma_1 &= \bk{a,b,p,q,t \mid abpa^{-1}b^{-1}q,
            ptqt^{-1}}*_{t=s^2}\bk{s^2}\\
          & \approx \bk{a,b,p,q,s \mid abpa^{-1}b^{-1}q, ps^2qs^{-2}} =\Sigma
\end{align*}
to embed a bad ladder into a the Cayley graph corresponding to the
presentation $\Sigma$. Since $\Sigma_0$ is a closed surface group,
therefore one-ended, and $\bk{s}$ cannot act with an infinite orbit on
a tree if $s^2$ fixes a point, \cite[Theorem 3.1]{touikan2015one}
implies that $\Sigma$ is one-ended. In particular $\Sigma$ is not
free. Hyperbolicity follows from the combination theorems
\cite{be-fe-comb,be-fe-comb-corr,k-m-free-constr}.

Again the universal cover is a tree of spaces obtained by gluing
hyperbolic planes $\bbH^2$ along $s$-lines and the proof goes
similarly to Proposition \ref{prop:ladders-are-convex}. The first
claim goes through as is, we leave the proof of Claim 2 (convexity of
$pq$-ladders) as an exercise in small cancellation theory (one can use
\cite[Theorem 9.4]{wise-mccammond-fans-ladders}.)

\subsection{Cubulating bad ladders}

A bad ladder consists of a chain of hexagons glued along edges. As an
illustration of \cite[Lemma 1.3]{huang2017hyperfiniteness}, observe
that if we cubulate a bad ladder, i.e. make it into a cube complex,
(see Figure \ref{fig:cubulated-ladder}) the conclusion of Proposition
\ref{prop:ladders-are-bad} no longer holds.
\begin{figure}[htb]
  \centering
  \includegraphics[width=0.8\textwidth]{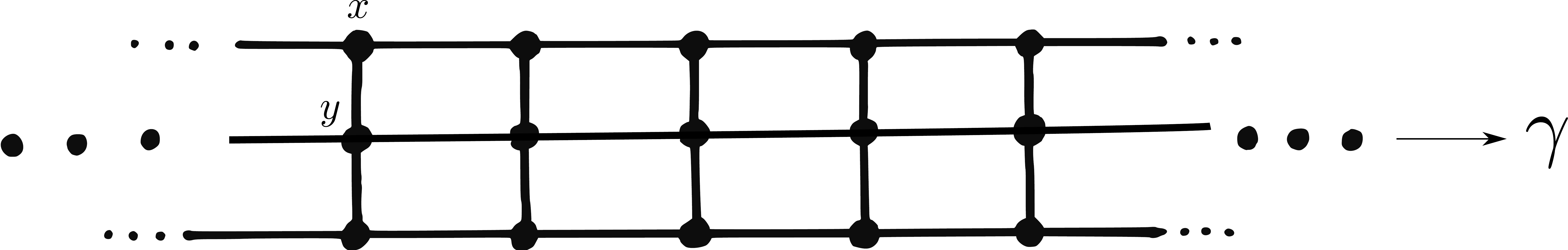}
  \caption{A cubulated bad ladder}
  \label{fig:cubulated-ladder}
\end{figure}
In fact both groups shown in this paper can be cubulated and therefore
do not give counterexamples to the conjecture that the action of
\emph{every} hyperbolic group $\Gamma$ on $\bdy\Gamma$ is hyperfinite,
a conjecture that this author believes to be true.

\bibliographystyle{alpha} \bibliography{biblio}

\def\cprime{$'$}
\begin{thebibliography}{GdlH90}

\bibitem[Aea91]{notes-on-hyp-groups}
J.~M. Alonso and et~al.
\newblock Notes on word hyperbolic groups.
\newblock In {\em Group theory from a geometrical viewpoint ({T}rieste, 1990)},
  pages 3--63. World Sci. Publ., River Edge, NJ, 1991.
\newblock Edited by H. Short.

\bibitem[BF92]{be-fe-comb}
M.~Bestvina and M.~Feighn.
\newblock A combination theorem for negatively curved groups.
\newblock {\em J. Differential Geom.}, 35(1):85--101, 1992.

\bibitem[BF96]{be-fe-comb-corr}
Mladen Bestvina and Mark Feighn.
\newblock Addendum and correction to: ``{A} combination theorem for negatively
  curved groups'' [{J}. {D}ifferential {G}eom. {\bf 35} (1992), no. 1, 85--101;
  {MR}1152226 (93d:53053)].
\newblock {\em J. Differential Geom.}, 43(4):783--788, 1996.

\bibitem[DJK94]{dougherty1994structure}
Randall Dougherty, Steve Jackson, and Alexander~S Kechris.
\newblock The structure of hyperfinite borel equivalence relations.
\newblock {\em Transactions of the American Mathematical Society},
  341(1):193--225, 1994.

\bibitem[GdlH90]{ghys1990groupes}
{\'E}tienne Ghys and Pierre de~la Harpe.
\newblock {\em Sur les groupes hyperboliques d'apres {M}ikhael {G}romov},
  volume~83 of {\em Progress in Mathematics}.
\newblock Birkh{\"a}user Boston, Inc., Boston, MA, 1990.

\bibitem[HSS17]{huang2017hyperfiniteness}
Jingyin Huang, Marcin Sabok, and Forte Shinko.
\newblock Hyperfiniteness of boundary actions of cubulated hyperbolic groups.
\newblock {\em arXiv preprint arXiv:1701.03969}, 2017.

\bibitem[KB02]{kapovich2002boundaries}
Ilya Kapovich and Nadia Benakli.
\newblock Boundaries of hyperbolic groups.
\newblock {\em Combinatorial and geometric group theory (New York,
  2000/Hoboken, NJ, 2001)}, 296:39--93, 2002.

\bibitem[KM98]{k-m-free-constr}
O.~Kharlampovich and A.~Myasnikov.
\newblock Hyperbolic groups and free constructions.
\newblock {\em Trans. Amer. Math. Soc.}, 350(2):571--613, 1998.

\bibitem[MW02]{wise-mccammond-fans-ladders}
Jonathan~P. McCammond and Daniel~T. Wise.
\newblock Fans and ladders in small cancellation theory.
\newblock {\em Proc. London Math. Soc. (3)}, 84(3):599--644, 2002.

\bibitem[Ser03]{serre}
Jean-Pierre Serre.
\newblock {\em Trees}.
\newblock Springer Monographs in Mathematics. Springer-Verlag, Berlin, 2003.
\newblock Translated from the French original by John Stillwell, Corrected 2nd
  printing of the 1980 English translation.

\bibitem[SW79]{scott-wall}
Peter Scott and Terry Wall.
\newblock Topological methods in group theory.
\newblock In {\em Homological group theory ({P}roc. {S}ympos., {D}urham,
  1977)}, volume~36 of {\em London Math. Soc. Lecture Note Ser.}, pages
  137--203. Cambridge Univ. Press, Cambridge, 1979.

\bibitem[Tou15]{touikan2015one}
Nicholas Touikan.
\newblock On the one-endedness of graphs of groups.
\newblock {\em Pacific Journal of Mathematics}, 278(2):463--478, 2015.

\end{thebibliography}

\end{document}